\documentclass[10pt,twocolumn,letterpaper]{article}

\usepackage{iccv}
\usepackage{times}
\usepackage{epsfig}
\usepackage{graphicx}
\usepackage{amsmath}
\usepackage{amssymb}
\usepackage[utf8]{inputenc}

\usepackage{xcolor}
\usepackage{amsthm}
\usepackage{framed}
\usepackage[pagebackref=true,breaklinks=true,colorlinks,bookmarks=false]{hyperref}

\iccvfinalcopy 

\def\iccvPaperID{6258} 

\newcommand{\R}{\mathbb{R}}

\DeclareMathOperator*{\argmin}{\arg \min}%

\DeclareMathOperator*{\st}{ \quad \textnormal{ s.t. }}%
\newtheorem{lemma}{Lemma}
\newtheorem{proposition}{Proposition}
\newtheorem{example}{Example}
\newtheorem{corollary}{Corollary}
\theoremstyle{definition}
\newtheorem{definition}{Definition}
\theoremstyle{remark}
\newtheorem*{remark}{Remark}

\newcommand{\importantBox}[1]{%
	\vspace{0pt}
	\begin{center}%
		\setlength{\fboxrule}{2pt}%
		\fbox{%
			\begin{minipage}{0.45\textwidth}%
				#1
			\end{minipage}%
		}%
	\end{center}%
	\vspace{0pt}
}

\begin{document}

\title{Parametric Majorization for Data-Driven Energy Minimization Methods}

\author{Jonas Geiping \qquad Michael Moeller\\
Department of Electrical Engineering and Computer Science, University of Siegen\\
{\tt\small \{jonas.geiping, michael.moeller\}@uni-siegen.de}
}

\maketitle

\begin{abstract}
Energy minimization methods are a classical tool in a multitude of computer vision applications. While they are interpretable and well-studied, their regularity assumptions are difficult to design by hand. Deep learning techniques on the other hand are purely data-driven, often provide excellent results, but are very difficult to constrain to predefined physical or safety-critical models. A possible combination between the two approaches is to design a parametric energy and train the free parameters in such a way that minimizers of the energy correspond to desired solution on a set of training examples. Unfortunately, such formulations typically lead to bi-level optimization problems, on which common optimization algorithms are difficult to scale to modern requirements in data processing and efficiency. 
In this work, we present a new strategy to optimize these bi-level problems. We investigate surrogate single-level problems that majorize the target problems and can be implemented with existing tools, leading to efficient algorithms without collapse of the energy function. This framework of strategies enables new avenues to the training of parameterized energy minimization models from large data.
\end{abstract}

\section{Introduction}

Energy minimization methods, also referred to as variational methods, are a classical tool in computer vision \cite{rudin_nonlinear_1992,chan_active_2001,cremers_convex_2011,chambolle_introduction_2010}. The idea is to define a data-dependent cost function $E$ that assigns a value to each candidate solution $x$. The desired optimal solution is then the target solution with the lowest energy value. This methodology has several advantages, for one, it is characterized by an \emph{explicit model} - namely the energy function to be minimized - and an implicit inference method - how we compute the minimizer of this energy is a separate problem. This duality allows a fruitful analysis, leading to controllable methods with provable guarantees that are paramount in many critical applications \cite{rohlfing_volume-preserving_2003,ritschl_improved_2011,zhan_safe_2017}. Furthermore, explicit knowledge over the model structure allows for explainable and clear modifications when the method is applied in a related task \cite{chen_insights_2014}.

Conversely, deep learning approaches \cite{lecun_deep_2015}, specifically deep feed-forward neural networks work by very different principles. The methodology of deep learning is characterized by \emph{implicit} models and explicit inference. The solution to the problem at hand is given directly by the output of the learned feed-forward structure. This is advantageous in practice and crucial for the efficient training of neural networks, however the underlying model of the problem structure is now only implicitly contained in the responses of the network. Deep neural networks have fundamentally changed the state-of-the-art in various computer vision applications, due to these properties as the inference operations are learned directly from large amounts of training data. These approaches are able to learn expressive and convincing mechanisms, examples of which can be found not only in recognition tasks (e.g. \cite{krizhevsky_imagenet_2012}), but also in denoising \cite{zhang_beyond_2017}, optical flow \cite{mayer_large_2016,ilg_flownet_2017} or segmentation tasks \cite{long_fully_2015,ronneberger_u-net:_2015,chen_semantic_2015}. Yet, as the underlying model is only implicitly defined and 'hidden' in the network structure, it is difficult to modify it for applications in other domains or to guarantee specific outputs. Domain adaptation is still an active field of research and several examples, for instance in medical imaging \cite{antun_instabilities_2019,finlayson_adversarial_2018}, have demonstrated the need for possibly model-based physically plausible output restrictions. This problem is most strikingly demonstrated by the phenomenon of adversarial examples \cite{szegedy_intriguing_2013} - the existence of input data, that, when fed through the network, leads to highly erroneous solutions. While one would expect that such behaviour is possibly unavoidable in recognition tasks \cite{shafahi_are_2018,metzen_universal_2017-1}, it should not be a factor in low-level computer vision applications.

Reviewing these two methodologies, we would - of course - prefer to have the best of both worlds. We would like to use both the large amounts of data at our disposal and our far-reaching domain knowledge in many tasks to train explicit models with a significant number of free parameters, so that their optimal solutions are similar to directly trained feed-forward networks.

A promising candidate for such a combination of learning- and model based approaches are \textit{parametrized energy minimization methods}. The idea of such methods is to define an energy $E$ that depends on the candidate solutions $x$, the input data $y$ and parameters $\theta$,
\begin{align}\label{eq:energy}
\begin{split}
E:\R^n \times \R^m \times \R^s ~&\to \R,\\
(x,y,\theta) ~&\mapsto E(x,y,\theta),
\end{split}
\end{align}
such that for a good choice of parameters $\theta$, the argument $x(\theta)=\arg\min_x E(x,y,\theta)$ that minimizes the energy over all $x$ is as close a possible to the desired true solution $x^*$. 


To train such parametric energies, assume we are given $N$ training samples $\lbrace(x^*_i,y_i)\rbrace_{i=1}^N$ and a continuous \emph{higher-level} loss function $l : \R^n \times \R^n \to \R$, which measures the deviation of solutions of the model to the given training samples. Determining the optimal parameters $\theta$ then becomes a \textit{bi-level optimization problem} combining both the higher-level loss function and the lower-level energy,
\begin{align}
\label{eq:bilevelProblemUpper}
   & \min_{\theta \in \R^s} \sum_{i=1}^N l(x_i^*,x_i(\theta)),  \\
   \label{eq:bilevelProblemLower}
    \textnormal{subject to }\qquad & x_i(\theta) = \argmin_{x \in \R^n} E(x, y_i,\theta).
\end{align}
Usual first-order learning methods are difficult to apply in this setting. For every gradient computation it is necessary to compute a derivative of the $\argmin$ operation of the lower-level problem, which is even further complicated if we consider parametrized non-smooth energy models which are wide-spread in computer vision \cite{cremers_convex_2011,chambolle_introduction_2010}.

Therefore, the goal of this paper is to analyze bi-level optimization problems and identify strategies that allow for efficient approximate solutions. We investigate 
single-level minimization problems with simple constraints without second-order differentiation, which are applicable even to non-smooth energies. Such forms allow scaling the previously limited training of energy minimization methods in computer vision to larger datasets and increase the effectiveness in applications where it is critical that the solution follows a specific model structure.


In the remainder of this paper we analyze the bi-level optimization problem to develop a rigorous understanding of sufficient conditions for a single-level surrogate strategy for continuous loss functions $l$ and convex, non-smooth lower-level energies $E$ to be successful. We introduce the concept of a \emph{parametric majorization function}, show relations to structured support vector machines and provide several levels of parametric majorization functions with varying levels of exactness and computational effort. We extend our approximations to an iterative scheme, allowing for repeated evaluations of the approximation, before illustrating the proposed strategies in computer vision applications. 

\section{Related Work}
%
%
%
%

The straightforward way of optimizing bi-level problems is to consider \emph{direct descent methods} \cite{kolstad_derivative_1990,savard_steepest_1994,colson_overview_2007}. These methods directly differentiate the higher-level loss function with respect to the minimizing argument and descend in the direction of this gradient. An incomplete list of examples in image processing is \cite{calatroni_bilevel_2017,chen_insights_2014,chen_learning_2012,chen_bi-level_2014-1,de_los_reyes_structure_2016,de_los_reyes_bilevel_2017,gould_differentiating_2016,hintermuller_optimal_2017,hintermuller_optimal_2017-1}. This strategy requires both the higher- and lower-level problems to be smooth and the minimizing map to be invertible. This is usually facilitated by implicit differentiation, as discussed in \cite{samuel_learning_2009,kunisch_bilevel_2013,chen_bi-level_2014-1,chen_insights_2014}.
In more generality, the problem of directly minimizing $\theta$ without assuming that smoothness in $E$ leads to optimization problems with equilibrium constraints (MPECs), see \cite{bennett_bilevel_2008} for a discussion in terms of machine learning or  \cite{dempe_is_2012,dempe_foundations_2002,dempe_bilevel_2015} and  \cite{colson_overview_2007}. 
This approach also applies to the optimization layers of \cite{amos_optnet:_2017-1}, which lend themselves well to a reformulation as a bi-level optimization problem.

\emph{Unrolling} is a prominent strategy in applied bi-level optimization across fields, i.e.  MRF literature \cite{barbu_learning_2009,martins_polyhedral_2009} in deep learning \cite{zheng_conditional_2015,chen_deeplab:_2016,chandra_fast_2016-1,lin_efficient_2016} and in variational settings \cite{ochs_bilevel_2015,larsson_revisiting_2018,larsson_projected_2018,hammernik_learning_2018,hammernik_deep_2017,riegler_atgv-net:_2016}. The problem is transformed into a single level problem by choosing an optimization algorithm $\mathcal{A}$ that produces an approximate solution to the lower level problem after a fixed number of iterations. $x(\theta)$ is then replaced by $\mathcal{A}(y,\theta)$. Automatic differentiation \cite{griewank_evaluating_2000} allows for an efficient evaluation of the gradient of the upper-level loss w.r.t to this reduced objective
\begin{equation}
	 \min_\theta \sum_{i=1}^N l(x_i^*,\mathcal{A}(y_i,\theta)).
\end{equation}
In general these strategies are very successful in practice, \emph{because} they combine the model and its optimization method into a single feed-forward process, where the model is again only implicitly present. 
Later works \cite{chen_learning_2015,chen_trainable_2017, hammernik_learning_2018,hammernik_deep_2017} allow the lower-level parameters to change in between the fixed number of iterations, leading to structures that model differential equations and stray further from underlying modelling. As pointed out in \cite{klatzer_learning_2016}, these strategies are more aptly considered as a set of nested quadratic lower-level problems.

Several techniques have been developed in the field of structured support vector machines (SSVMs) \cite{taskar_max-margin_2004, collins_discriminative_2002, altun_hidden_2003, tsochantaridis_large_2005} that are very relevant to the task of learning energy models, as SSVMs can be understood as bi-level problems with a lower-level energy that is linear in $\theta$ and often a non-continuous higher-level loss. Various strategies such as  margin rescaling \cite{taskar_max-margin_2004}, slack rescaling \cite{tsochantaridis_large_2005, wiseman_sequence--sequence_2016}, softmax-margins \cite{gimpel_softmax-margin_2010} exist and have also been applied recently in the training of computer vision models in \cite{knobelreiter_end--end_2017,colovic_end--end_2017}, we will later return to their connection to the investigated strategies.
\section{Bi-Level Learning}

We now formalize our learning problem. 
We assume the lower-level energy $E$ from~\eqref{eq:energy} to be convex (but not necessarily smooth) in its first variable $x \in \R^n$ and to depend continuously on input data $y \in \R^m$ and parameters $\theta \in \R^s$. We assume its minimizer $x(\theta)$ to be unique.
For our higher-level loss function \eqref{eq:bilevelProblemUpper} $l : \R^n \times \R^n \to \R$, we assume that it fulfills $l(x,y) \geq 0, l(x,x) = 0$ for all $x,y$ and is differentiable in its second argument. 

Note that this formulation of bi-level optimization problems directly generalizes classical supervised (deep) learning with a network $\mathcal{N}(\theta, y)$ via the quadratic energy $E(x,y_i,\theta) = \frac{1}{2}||x-\mathcal{N}(\theta, y_i)||^2$, for which $x_i(\theta) = \mathcal{N}(\theta, y_i)$.

\emph{Preliminaries (Convex Analysis):} Let us summarize our notation and some fundamental results from convex analysis. We refer the reader to \cite{bauschke_convex_2011} for more details. 
We denote by $\partial E(x)$ the set of subgradients of a convex function $E$ at $x$. We define the Bregman distance between two vectors relative to a convex function $E$ by $D_E^p(x,y) = E(x) - E(y) - \langle p, x-y \rangle$ for a subgradient $p \in \partial E(y)$, intuitively the Bregman distance measures the difference of the energy at $x$ to its linear lower bound around $y$. $E^*(p) = \sup_x \langle p,x\rangle -E(x)$ is the convex conjugate of $E$. $x$ is a minimizer of the energy $E$ if and only if $0 \in \partial E(x)$ or equivalently by convex duality $x \in \partial E^*(0)$.
$E$ is $m$-strongly convex if $D_E^p(x,y) \geq \frac{m}{2}||x-y||^2$ for all $x,y$. Conversely, if $E$ is $m$-strongly convex, then $E^*$ is $\frac{1}{m}$-strongly smooth, i.e. $D_{E^*}(p,q) \leq \frac{2}{m}||p-q||^2$. Furthermore $D_E^p(x,y) = D_{E^*}^x(p,q), q \in \partial E(x)$ holds for all Bregman distances \cite{burger_bregman_2016}. We consider parametrized energies in several variables, yet we always assume (sub)-gradients, Bregman distances and convex conjugates to be with respect to the first argument $x$.

%

\subsection{Majorization of Bi-level Problems}
As previously discussed, directly solving the bi-level problem as posed in Eq.~\eqref{eq:bilevelProblemUpper} and \eqref{eq:bilevelProblemLower} is tricky. We need to implicitly differentiate the minimizing argument $x_i(\theta)$ for all $N$ samples just to apply a first-order method in $\theta$ - which is in stark contrast to our goal of finding efficient and scalable algorithms.

Let us 
instead look at the problem from a very different angle and entertain the idea that the loss function $l$ is actually of secondary importance to us. We really only want to find parameters $\theta$ so that our training samples are well reconstructed, $x_i^* \approx x_i(\theta)$. If we go so far as to assume that the loss value of our optimal parameters $\theta^*$ is zero, meaning that minimizers of our energy are perfectly able to reconstruct our training samples, then the bi-level problem is reduced to a single-level problem, inserting $x_i^* = x_i(\theta^*)$:
\begin{equation}\label{eq:single_level_gradient}
    \min_\theta \st 0 \in \partial E(x_i^*,y_i,\theta),
\end{equation}
which we could solve via
\begin{equation}\label{eq:gradient_penalty}
    \min_\theta \sum_{i=1}^N ||q_i||^2 \st q_i \in \partial E(x_i^*,y_i,\theta)
\end{equation}
This train of thought is closely interconnected to the notion of separability in Support Vector Machine methods \cite{vapnik_statistical_1998}, where it is assumed that given training samples are linearly separable, which is equivalent to assuming that the classification loss is zero on the training set.

However minimizing Eq.~\eqref{eq:gradient_penalty} is often not a good choice. A simple example is $E(x,y,\theta) = (\theta x -y)^2$, i.e. we simply try to learn a positive scaling factor $\theta$ between $x$ and $y$. Problem \eqref{eq:single_level_gradient} can then be written as $ \min_\theta \sum_i (\theta^2 x_i^* - \theta y_i)^2$ and is trivially minimized by $\theta=0$. Such a solution makes $E$ independent of $x$ such that every $x$ becomes a minimizer. This phenomenon is referred to as \emph{collapse} of the energy function \cite{lecun_loss_2005,lecun_tutorial_2006} in machine learning literature, and clearly cannot be a good strategy to learn a scaling factor. 


Interestingly, the scaling problem can be reformulated into a reasonable (non-collapsing) problem, if we require \eqref{eq:gradient_penalty} to \textit{majorize} the bilevel problem: If we consider the higher-level loss function $l(x_i^*,x_i(\theta)) = (x_i^*-x_i(\theta))^2$, then our surrogate problem $\sum_i (\theta^2 x_i^* - \theta y_i)^2$ is clearly not a majorizer for arbitrary $\theta$. However, if we consider a reformulation of the energy to $E(x) = (x - \frac{1}{\theta}y)^2$, then this reformulation leads to a \emph{majorizing} surrogate $ \sum_i (x_i^* - \frac{1}{\theta}y_i)^2$. Minimizing $\theta$ now leads to learning the desired scaling factor.


Our toy example motivates us to formalize the concept of majorizing surrogates:
\begin{definition}[Parametrized Majorizer]\label{def:maj}
    Given a bi-level optimization problem in the higher level loss $l(x,y)$ and lower-level energy $E(x,y,\theta)$, we call the function $S(x, y, \theta) : \R^n \times \R^m \times \R^s \to \R$ a parametrized majorizer, if
    \begin{align*}
        &\forall \theta \in \R^s:  &&l(x,x(\theta)) \leq S(x, y, \theta) \quad \\
        &\forall \theta \in \R^s \st &&l(x,x(\theta)) = 0 \implies S(x,y, \theta) = 0
    \end{align*}
    hold for any $x,y \in \R^n \times \R^m$.
\end{definition}

This definition allows us to formalize our objective further. We investigate replacing the bi-level optimization problem \eqref{eq:bilevelProblemUpper}, \eqref{eq:bilevelProblemLower} by the minimization of a suitable parametrized majorizer, i.e.  
\begin{align}
\label{eq:ourApproachGeneric}
\min_{\theta \in \mathbb{R}^s} \sum_{i=1}^N S(x_i^*,y_i,\theta). 
\end{align}
An immediate conclusion of Definition \ref{def:maj} is that the function $S$ now certifies our progress as $S(x,y, \theta) = 0$ implies $l(x,x(\theta)) = 0$. Moreover, our goal is to choose majorizers $S$ in such a way that they yield \textbf{single-level} problems \eqref{eq:ourApproachGeneric}, meaning it is not necessary to differentiate an $\argmin$ operation to minimize them or to solve an equally difficult reformulation, making them significantly easier to solve. 
\subsection{Single-Level Majorizers}\label{sec:single-level}
One possible way to find a majorizer that satisfies the previously postulated properties is by considering the majorizer naturally induced through the Bregman distance of the lower level energy. We assume the following condition
\begin{equation} \label{eq:blanket_assumption}
    l(x,z) \leq D_{E_\theta}(x,z) \quad \forall x,z \in \R^n, \theta \in \R^s,
\end{equation}
and propose the surrogate problem
\begin{equation}\label{eq:surrogate_1}
    \min_\theta ~ \sum_{i=1}^N D_{E_\theta} \left( x_i^*,x_i(\theta) \right).
\end{equation}
Condition \eqref{eq:blanket_assumption} is an assumption on both the loss function and the energy. It thus delineates the class of bi-level problems that can be attacked with this majorization strategy. However this condition is quite general. For a large class of loss functions, we only need the energy to contain a term that also induces the loss function, a property also known as (relative) strong convexity \cite{teboulle_simplified_2018,lu_relatively_2018}:
\begin{proposition}\label{prop:relatively_convex_cond}
If the loss function $l(x,y)$ is a Bregman distance induced by a strictly convex function $w:\R^n \to \R$, i.e. $l(x,y) = D_w(x,y)$, then assumption \eqref{eq:blanket_assumption} is fulfilled if the energy $E$ is $w$-strongly convex, i.e. if $E(x) -w(x)$ is still a convex function.
\end{proposition}
\emph{Proof:} We write $E$ as $E(x) = \hat{E}(x) + w(x)$ and apply the additive separability of Bregman distances to find $D_E(x,y) = D_{\hat{E}}(x,y) + D_w(x,y)$, which is greater than or equal to $D_w(x,y)$, as $D_{\hat{E}}(x,y)$ is non-negative due to the convexity of $\hat{E}$.
For the usual euclidean loss, this property reduces to strong convexity:
\begin{example}\label{prop:strongly_convex_cond}
    If the loss function is given by a squared Euclidean loss, $l(x,y)=\frac{1}{2}||x-y||^2$ and the energy is $m$-strongly convex, then assumption \eqref{eq:blanket_assumption} is fulfilled for the energy $\frac{1}{m} E$.
\end{example}

The question remains whether the proposed surrogate problem \eqref{eq:surrogate_1} is efficiently solvable. We especially wanted to circumvent the differentiation of $x(\theta)$. However $D_E \left( x_i^*,x_i(\theta) \right)$ is much easier to solve, in comparison to the original bi-level problem, as we can see in both its primal and its dual formulation. First, from a primal viewpoint, we have
\begin{align*}
    &D_E \left( x_i^*,x_i(\theta) \right) \\
    =& E(x_i^*,y_i,\theta) - E(x_i(\theta),y_i,\theta) - \langle p_i,x_i^* - x_i(\theta)\rangle,
\end{align*}
for some subgradient $p_i \in \partial E(x_i(\theta))$ which we have not specified yet. But, as $0 \in \partial E(x_i(\theta))$ as $x_i(\theta)$ is by definition a solution to the lower-level problem, we may take $p=0$ and simplify to
\begin{align*}
E(x_i^*,y_i,\theta) - E(x_i(\theta),y_i,\theta).
\end{align*}
Now $x_i(\theta)$ is contained solely in $E$ and we can write
\importantBox{\textbf{Bregman Surrogate:}
\begin{align}\label{eq:sur_primal}
D_{E_\theta}^0 \left( x_i^*,x_i(\theta) \right) = \max_{x \in \R^n} E(x_i^*,y_i,\theta) - E(x,y_i,\theta).
\end{align}}
This surrogate function is already much simpler than the original bi-level problem. We can minimize \eqref{eq:sur_primal} either by alternating minimization in $\theta$ and maximization in $x$ or by jointly optimizing both variables. However, the problem is still set up as a saddle-point problem which is not ideal for optimization.
\begin{remark}
Interestingly, this discriminative formulation is not wholly unfamiliar. We can understand this as an appropriate generalization of generalized perceptron training \cite{lecun_loss_2005,lecun_tutorial_2006,tappen_learning_2007} as discussed as far back as \cite{rosenblatt_perceptron:_1958}. See the appendix for further details. In vein of this comparison, conditions 1 and 2 from e.g. \cite{lecun_loss_2005}, i.e. conditions on the existence of a margin between the optimal solution and other candidate solutions central to (S)SVM methods \cite{vapnik_statistical_1998, taskar_learning_2005, taskar_structured_2006} are reflected in Proposition \ref{prop:relatively_convex_cond} in the convex continuous setting. Due to continuity of the energy and loss function we cannot obey a fixed margin, yet we impose that the energy grows at least as fast as the loss function, when we move away from the optimal solution.
\end{remark}
We can resolve the saddle-point question by analyzing the surrogate \eqref{eq:surrogate_1} from a dual standpoint, as by Bregman duality \cite{benning_modern_2018-1}
\begin{equation}\label{eq:dual_form}
    D_{E_\theta}^0 \left( x_i^*,x_i(\theta) \right) = D_{E^*_\theta}^{x_i^*}(0, q_i)
\end{equation}
for $q_i \in \partial E(x_i^*,y,\theta)$.  Contrasting this formulation with our initial goal of penalizing the subgradient as in Eq.~\eqref{eq:gradient_penalty}, we see that the Bregman distance induced by $E^*$ is the natural 'distance' by which to penalize the subgradient in the sense that penalizing the subgradient at $x_i^*$ with this generalized distance recovers a majorizing surrogate.

We can further simplify the dual formulation by applying Fenchel's theorem:
\begin{equation}\label{eq:primaldual}
D_{E^*_\theta}^{x_i^*}(0, q_i) = E(x_i^*,y_i,\theta) + E^*(0,y_i,\theta).
\end{equation}
Computing $E^*(0)$ is exactly as difficult as minimizing $E$ (as $E^*(0) = \min_x E(x)$), so we need to rewrite this surrogate in a tractable manner. To do so, we assume that $E$ can be additively decomposed into two parts,
\begin{equation}
     E(x,y,\theta) = E_1(x,y,\theta) + E_2(x,y,\theta),
\end{equation}
where both $E_1$ and $E_2$ are convex in their first argument and their convex conjugates are simple to compute. 
%
Exploiting that $E^*(0) = \min_z E_1^*(-z) + E_2^*(z)$ yields
\begin{equation}\label{eq:sur_dual}
    D_{E^*_\theta}^{x_i^*}(0, q_i) = \min_{z \in \R^n} E(x_i^*,y_i,\theta) + E_1^*(-z,y,\theta) + E_2^*(z,y,\theta).
\end{equation}
In comparison to the primal formulation in Eq~\eqref{eq:sur_primal}, we have now reformulated the problem from a saddle point problem (minimizing in $\theta$ and maximizing in $x$) to a pure minimization problem, which is easier to handle. This is a generalization of the dual formulation discussed in the linear context of SSVMs for example in \cite{taskar_learning_2005,taskar_structured_2006}.

However for both variants we still need to handle an auxiliary variable. We can trade some of this computational effort for a weaker majorizer by making specific choices for $z$ in Eq.~\eqref{eq:sur_dual}. To illuminate these choices we introduce the function $W_E(p,x) = E^*(p) + E(x) - \langle p,x\rangle$ \cite{reich_existence_2011,butnariu_proximal-projection_2008}, which allows us to write 
\begin{equation}\label{eq:W_E}
    D_{E^*_\theta}^{x_i^*}(0, q_i) = \min_{z \in \R^n} W_{E_1,\theta}(-z,x_i^*) + W_{E_2,\theta}(z,x_i^*).
\end{equation}
Note that $W_E(p,x) = 0$ if $p \in \partial E(x)$. As such choosing either $-z \in \partial E_1(x_i^*)$ or $z \in \partial E_2(x_i^*)$ allows us to simplify the problem further. This is especially attractive if $E$ is differentiable, as then both surrogates can be computed without auxiliary variables. We will denote these as \emph{partial} surrogates, owing to the fact that we minimize only one term in \eqref{eq:W_E}
\importantBox{\textbf{Partial Surrogate:}
\begin{equation}\label{eq:partial_surrogate}
    \min_{z \in \partial E_2(x_i^*,y_i,\theta)} W_{E_1,\theta}(-z,x_i^*).
\end{equation}
}
Effectively, this reduces the requirements of \eqref{eq:sur_dual}, as only the convex conjugate of $E_1$ needs to be computed. By symmetry, the other partial surrogate follows analogously.

We can finally also return to the previously discussed gradient penalty \eqref{eq:gradient_penalty}. If our energy $E$ is $m(\theta,y)$-strongly convex, then its convex conjugate is strongly smooth and we can bound the dual formulation \eqref{eq:dual_form} via 
\importantBox{\textbf{Gradient Penalty}
\begin{equation}\label{eq:gradient_penalty_corrected}
    \frac{1}{m(\theta,y_i)} ||q_i||^2 \st q_i \in \partial E(x_i^*,y_i,\theta).
\end{equation}}
While this formulation allows us to minimize an upper bound on the bi-level problem without either auxiliary variables or knowledge about $E_1^*$ or $E_2^*$, it also is the crudest over-approximation among the considered surrogates as the following proposition illustrates.
\begin{figure}
    \centering 
    \includegraphics[width=0.5\textwidth]{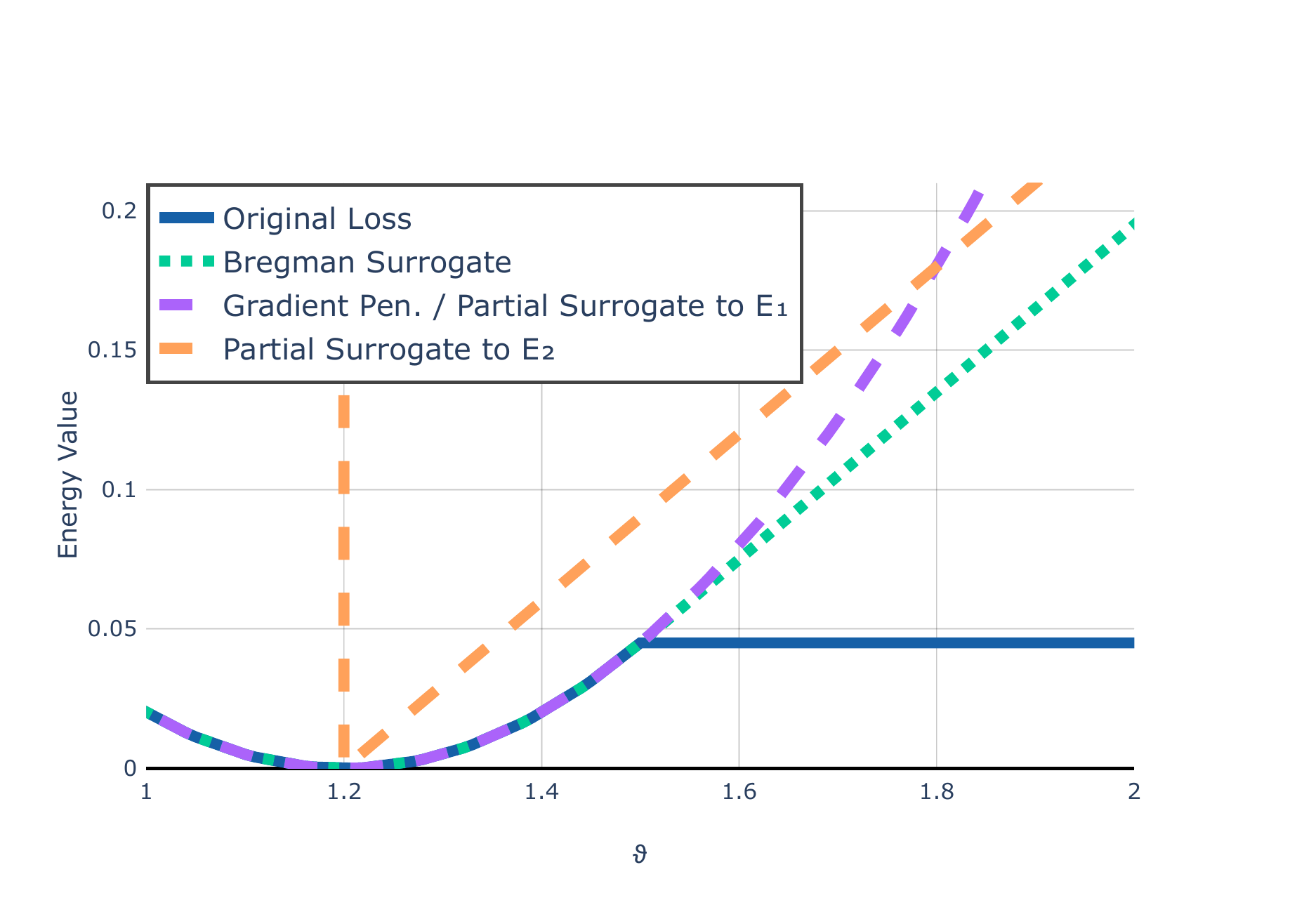}
    \caption{Visualization of surrogate functions for the bi-level problem given in Eq.~\eqref{eq:bilevel_1d}. The blue line marks the original bi-level problem, the green dots marks the Bregman distance surrogate discussed in Eq.~\eqref{eq:sur_primal}. The orange curve marks the partial surrogate obtained from ~\eqref{eq:W_E} by inserting $z = \nabla E_1(x^*)$, whereas the purple line marks the other partial surrogate \eqref{eq:partial_surrogate} which is equivalent to the gradient penalty \eqref{eq:gradient_penalty_corrected} here.
    \label{fig:1D_vis}}
\end{figure}
\begin{proposition}[Ordering of parametric majorizers]\label{prop:ordering}
    Assuming the condition $l(x,z) \leq D_{E_\theta}(x,z)$ from Eq.~\eqref{eq:blanket_assumption}, we find that the presented parametric majorizers can be ordered in the following way:
    \begin{align*}
         l(x_i^*,x(\theta)) &\leq D_{E_\theta}^0 \left( x_i^*,x_i(\theta) \right) = D_{E^*_\theta}^{x_i^*}(0, q_i) \\
         &\leq \min_{z \in \partial E_2(x_i^*)} W_{E_1}(-z,x_i^*) \\
         & \leq \frac{1}{m(\theta,y)} ||q_i||^2 \st q_i \in \partial E(x_i^*,y,\theta).
    \end{align*}
    The Bregman surrogate \eqref{eq:sur_primal} majorizes the original loss function and is in turn majorized by the partial surrogate \eqref{eq:partial_surrogate} which is majorized by the gradient penalty \eqref{eq:gradient_penalty_corrected} under the assumption of strong convexity.
\end{proposition}
\begin{proof} See appendix. \end{proof}
As a clarifying example, we can simplify these majorizers in the differentiable setting:
\begin{example}[Differentiable Energy]
    Let $E$ be differentiable and $m(\theta,y)$-strongly convex, 
    then the majorizers in Prop. \ref{prop:ordering} are given by
    \begin{align*}
         l(x_i^*,x(\theta)) &\leq D_{E_\theta} \left( x_i^*,x_i(\theta) \right) = D_{E^*_\theta}(0, \nabla E(x_i^*,y_i,\theta)) \\
         &\leq W_{E_1}(-\nabla E_2(x_i^*),x_i^*) \\ 
         & \leq \frac{1}{m(\theta,y)} ||\nabla E(x_i^*,y_i,\theta)||^2.
    \end{align*}
\end{example}
\subsection{Intermission: One-Dimensional Example}
Let us illustrate our discussion with a toy example. We consider the non-smooth bi-level problem of learning the optimal sparsity parameter $\theta$ in the bi-level problem:
\begin{align}
\label{eq:bilevel_1d}
   & \min_{\theta \in \R} \frac{1}{2}|x^*-x(\theta)|^2,  \\
   \label{eq:bilevel_1d_lower}
    \textnormal{subject to }\qquad & x(\theta) = \argmin_x \frac{1}{2}|x-y|^2 + \theta |x|.
\end{align}
As the lower-level energy is 1-strongly convex and the upper level loss is quadratic $l(x,y)\leq D_{E_\theta}(x,y)$ holds. Detailed derivations of all three surrogate functions of this example can be found in the appendix. Figure \ref{fig:1D_vis} visualizes these surrogates, plotting their energy values relative to $\theta$. Due to the low dimensionality of the problem, all surrogate functions coincide with the original loss function at the optimal value of $\theta$. It is further interesting to note that the Bregman surrogate is exactly identical with the original loss function in the vicinity of the optimal value, due to the low dimensionality of the example. 

\subsection{Iterative Majorizers}
We used subsection \ref{sec:single-level} to construct a series of upper bounds to facilitate a trade-off between efficiency and exactness. However what happens if we are not satisfied with the exactness of the Bregman surrogate \eqref{eq:surrogate_1}? This setting can happen especially if $x^*$ and $x(\theta)$ are significantly incompatible and subsequently $l(x^*,x(\theta))$ is large, even for optimal $\theta$. For example if we try to optimize only a few hyper-parameters we might not at all expect $x(\theta)$ to be close to $x^*$. This discussion can again be linked to the notion of ’separability’ in SVM approaches \cite{vapnik_statistical_1998}: The quality of the majorizing strategy is directly related to the level of ’separability’ of the bi-level problem.

However, we can use the previously introduced majorizers iteratively. To do so we need to develop a majorizer that depends on a given estimate $\bar{x}$.
\begin{proposition}
    Under the standing assumption that $l(x,y) \leq D_{E_\theta}(x,y)$ \eqref{eq:blanket_assumption} and if the loss function is induced by a strictly convex function $w:\R^n \to \R$, i.e. $l(x,y) = D_w(y,x)$, we have the following inequality:
    \begin{equation}
        l(x,y) \leq l(x,z) + \langle \nabla_z l(x,z), y- z \rangle + D_E(z,y).
    \end{equation}
\end{proposition}
\begin{proof}
It holds that $l(x,y) = D_w(y,x)$ which is equivalent to $D_w(y,z) + D_w(z,x) - \langle \nabla w(x) -\nabla w(z),z-y\rangle$ by the Bregman 3-Point inequality \cite{chen_convergence_1993,teboulle_simplified_2018}. Using the standing assumption and that $\nabla w(x) -\nabla w(z) = \nabla_x D_w(x,z)$ we find the proposed inequality.
\end{proof}
Assume we are given an estimated solution $\bar{x}_i$, then we can use this estimate to rewrite our bound to
\begin{align}\label{eq:direct_iterative}
\begin{split}
    l(x_i^*,x_i(\theta)) \leq &l(x_i^*,\bar{x}_i) + \langle \nabla l(x_i^*, \bar{x}_i), x_i(\theta)- \bar{x}_i \rangle \\ &+D_E(\bar{x}_i,x_i(\theta)).
\end{split}
\end{align}
This is a linearized variant of the parametric majorization bound and as such a nonconvex composite majorizer in the sense of \cite{geiping_composite_2018}, as such a key property of majorization-minimization techniques remains in the parametrized setting, choosing $\bar{x}_i = x_i(\theta^k)$:
\begin{proposition}[Descent Lemma]\label{prop:descent}
The iterative procedure given by repeatedly minimizing the right-hand side of Eq. \eqref{eq:direct_iterative} in $\theta$ and setting $\bar{x_i} = x_i(\theta^k)$ 
is guaranteed to be stable, i.e. not to increase the bi-level loss:
\begin{equation}
    \sum_{i=1}^N l \left( x_i^*,x_i(\theta^{k+1}) \right) \leq \sum_{i=1}^N l \left( x_i^*,x_i(\theta^{k}) \right) \tag{23}
\end{equation}
\end{proposition}
\begin{proof}
    See appendix.
\end{proof}
However this algorithm cannot be applied directly, as we would still need to differentiate $x_i(\theta)$ appearing in the linearized part. Nevertheless, we can use both Fenchel's inequality $\langle p,x \rangle \leq E(x) + E^*(p)$ and the previously established $D_{E_\theta}(x,x(\theta)) = E(x,y,\theta) - E(x(\theta),y,\theta)$ to find an over-approximation to the iterative majorizer of Prop. \ref{prop:descent}:
\begin{align*}
     &l(x_i^*,x_i(\theta)) \\
\leq ~ &l(x_i^*,\bar{x}_i) - \langle \nabla l(x_i^*, \bar{x}_i), \bar{x}_i \rangle \\
     &+ E^* \left( \nabla l(x_i^*, \bar{x}_i),y_i,\theta \right) + E(x_i(\theta), y_i, \theta) \\
     &+ E(\bar{x}_i,y_i,\theta) - E(x_i(\theta),y_i,\theta) \\
    = ~ & l(x_i^*,\bar{x}_i) - \langle \nabla l(x_i^*, \bar{x}_i), \bar{x}_i \rangle \\
   & + E(\bar{x}_i,y_i,\theta)+ E^* \left( \nabla l(x_i^*, \bar{x}_i),y_i,\theta \right) 
\end{align*}
This estimate reveals that we can approximate the iterative majorizer much like the previously discussed surrogates:
\importantBox{ \textbf{Iterative Surrogate}
\begin{equation}\label{eq:iterative_surrogate}
    E(\bar{x}_i,y,\theta)+ E^* \left( \nabla l(x_i^*, \bar{x}_i),y_i,\theta \right) + C,
\end{equation}
}
as the constant $C = l(x_i^*,\bar{x}_i) - \langle \nabla l(x_i^*, \bar{x}_i), \bar{x}_i \rangle$ does not depend on $\theta$. We essentially return to Eq.~\eqref{eq:primaldual} and only the input to $E$ and $E^*$ changes with respect to $\bar{x}_i$. This strategy recovers the previous majorizer as a special case:
\begin{corollary}\label{cor:gt_init}
    If we linearize around $\bar{x_i} = x_i^*$, then we recover the Bregman surrogate of \eqref{eq:surrogate_1}.
\end{corollary}
\begin{proof}  
    If $\bar{x_i} = x_i^*$, then $l(x_i^*,\bar{x}_i) = 0$ and $\nabla l(x_i^*,\bar{x}_i) = 0$ by the properties of the differentiable loss function. As such the constant term $C$ is zero and $E^* \left( \nabla l(x_i^*, \bar{x}_i),y_i,\theta \right) = E^*(0,y_i,\theta)$ so that we recover \eqref{eq:primaldual} which is equivalent to the Bregman surrogate \eqref{eq:surrogate_1}.
\end{proof}
We can use this surrogate to form an efficient approximation to a classical majorization-minimization strategy as in \cite{sun_majorization-minimization_2017,mairal_incremental_2015,mairal_optimization_2013-1,hunter_tutorial_2004}. Notably the 'tightness' of the majorization is violated by the over-approximation, i.e. inserting $\theta^k$ into the majorizer does not recover $l(x_i^*,x_i(\theta^k))$. We iterate
\begin{align}\label{eq:iterative_scheme}
\begin{split}
        \theta^{k+1} = \argmin_\theta \sum_{i=1}^N &E^* \left( \nabla l(x_i^*,x_i(\theta^k)),y_i,\theta\right)\\
    + &E \left(x(\theta^k),y_i,\theta \right)
\end{split}
\end{align}
As the application of this iterative scheme reduces to a simple change from Eq~\eqref{eq:primaldual} to Eq.~\eqref{eq:iterative_surrogate}, we can easily apply it in practice to further increase the fidelity of the surrogate by solving a sequence of fast surrogate optimizations. We initialize the scheme with $\bar{x}_i = x_i^*$ as suggested from Corollary \ref{cor:gt_init} and either stop iterating or reduce the step size of the surrogate solver if the higher-level objective is increased after an iteration.

\section{Examples}
This section will feature several experiments\footnote{An implementation of these experiments can be found at {\footnotesize \url{https://github.com/JonasGeiping/ParametricMajorization}}.} in which we will illustrate the application of the investigated methods. We will show two concepts of new applications that are possible in parametrized variational settings, \ref{sec:ct} and \ref{sec:segmentation}. 
We then show an application to image denoising in \ref{sec:analysis_operators}.

\subsection{Computed Tomography}\label{sec:ct}
Making only specific parts of a variational model learnable is especially interesting for computed tomography (CT). An image $x$ is to be reconstructed from data $y = Ax + n$ that is formed by applying the radon transform to the image $x$ and adding noise $n$. While first fully-learning based solutions to this problem exist (e.g. \cite{jin_deep_2017,kang_deep_2017}), suitable networks are difficult to find not only due to the ill-posedness of the underlying problem, but also due to the well-justified concerns about fully learning-based approaches in medical imaging \cite{antun_instabilities_2019}. To benefit from the explicit control of the data fidelity of the reconstruction,  we consider to introduce a learnable linear correction term into an otherwise classical reconstruction technique via
\begin{align*}
    x_i(\theta) = \arg \min_{ x} \frac{1}{2}\|Ax-y_i\|_2^2 + \beta R(x) 
    + \langle x, \mathcal{N}(\theta, y_i) \rangle,
\end{align*}
for a suitable network $\mathcal{N}$ (we chose 8 blocks of $3\times3$ convolutions with 32 filters, ReLU activations, and batch-normalization, and a final $5\times 5$ convolution), and $R$ denoting the Huber loss of the discrete gradient of $x$. 

As both convex conjugates are difficult to evaluate in closed-form, we choose the gradient penalty
\eqref{eq:gradient_penalty_corrected}, which is a parametric majorizer for euclidean loss if $A$ has full rank (and practically even works beyond this setting, as it majorizes $||A(x-y)||^2$ even for rank-deficient $A$). According to \eqref{eq:gradient_penalty_corrected} we consider
\begin{align*}\label{eq:ctPrimalSurrogate}
    \min_{\theta \in \R^s} \sum_{i=1}^n \|A^*Ax_i^* - A^*y_i + \beta \nabla R(x_i^*) + \mathcal{N}(\theta, y_i)\|_2^2,
\end{align*}
train on simulated noisy data and test our model on the widely-used Shepp-Logan phantom. Figure \ref{fig:ctExample} illustrates the 
the resulting reconstruction, as well as the best reconstruction using the variational approach without the additional linear correction term after a grid-search for the optimal $\beta$. As we can see, the surrogate trained the linear correction term well enough to improve the PSNR of the reconstruction by almost 2dB. Moreover, the influence of the linear correction term can still be visualized and the data fidelity can easily be controlled via a suitable weighting. We visualize the correction map in the appendix.

\begin{figure}
    \centering
        \begin{tabular}{cc}
    \includegraphics[height = 0.22\textwidth]{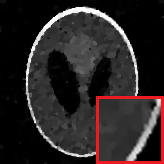} &
    \includegraphics[height = 0.22\textwidth]{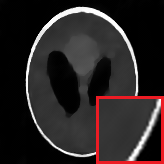}\\
\small  Huber-TV, PSNR 23.9 &\small  Learned cor., PSNR 25.8\\
    \end{tabular}
    \caption{Learning a linear correction term for a Huber-regularized CT reconstruction problem using the gradient penalty \eqref{eq:gradient_penalty}.  }
    \label{fig:ctExample}
\end{figure}


\subsection{Variational Segmentation}\label{sec:segmentation}
For a very different (and non-smooth) example, consider the task of learning a variational segmentation model \cite{chan_active_2001,chambolle_convex_2012,cremers_convex_2011,nieuwenhuis_survey_2013}. We are interested in learning a model whose minimizer coincides with a (semantic) segmentation of the input data. The lower-level problem is given by
\begin{equation}
    x(\theta) = \argmin_{x} -\langle \mathcal{N}(\theta,y),x\rangle + ||Dx||_1 + h(x),
\end{equation}
where $h(x) = \sum_{j=1}^n x_i\log(x_i) + I_\Delta(x)$ is the entropy function on the unit simplex $\Delta$ \cite{beck_mirror_2003}. $\mathcal{N}(\theta,y)$ is some parametrized function that computes the potential of the segmentation model, this can be a deep neural network, as we only require convexity in $x$ and not in $\theta$. $D$ is a finite-differences operator, so that the overall total variation (TV) term $||Dx||_1$ measures the perimeter of a segmentation $x$ if $x \in \lbrace 0,1\rbrace^n$. The entropy function crucially not only leads to a strictly convex model but also represents the structure of a usual learned segmentation method. Without the perimeter term, a solution to the lower-level problem would be given by
\begin{equation}\label{eq:segmentation_simple}
    x(\theta) = \nabla h^*( \mathcal{N}(\theta,y)).
\end{equation}
Due to \cite[P.148]{rockafellar_convex_1970}, $\nabla h^*$ is exactly the $\operatorname{softmax}$ function, so that Eq.~\eqref{eq:segmentation_simple} is equivalent to applying a parametrized function $\mathcal{N}$ and then applying the $\operatorname{softmax}$ function to arrive at the final output, a usual image recognition pipeline during training. As a higher-level loss, we choose $\operatorname{log}$ loss
\begin{equation}\label{eq:cross_entropy}
    \sum_{i=1}^N -\langle x_i^*,\log(x_i(\theta)) \rangle = \sum_{i=1}^N D_h(x_i^*,x_i(\theta))
\end{equation}
so that the bi-level problem without the perimeter term is equivalent to minimizing the cross-entropy loss of $\mathcal{N}(\theta,y)$. 
With the inclusion of the perimeter term, however, we cannot find a closed-form solution for $x(\theta)$ need to consider bi-level optimization.
\begin{figure}
    \centering
    \includegraphics[width=0.49\textwidth]{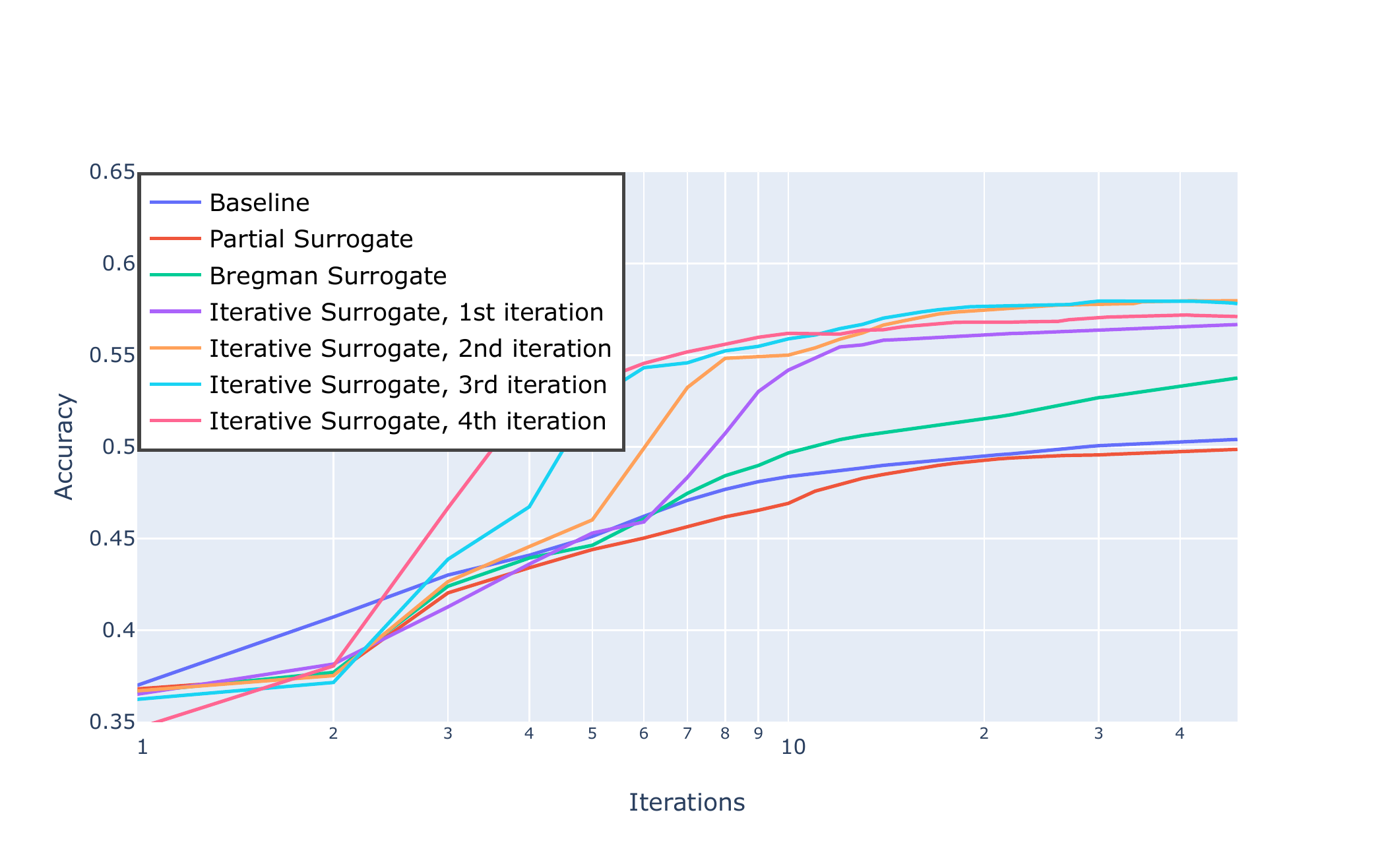}
    \caption{Training accuracy for the variational segmentation model discussed in Section \ref{sec:segmentation} for a linear model $\mathcal{N}(\theta,y_i)$. Directly training a cross-entropy loss without the perimeter term, training the Bregman surrogate Eq~\eqref{eq:bregman_segmentation}, the Partial surrogate Eq~\eqref{eq:partial_surrogate} and four iterations of the iterative scheme 
    are compared. We find that the end-to-end training with the perimeter term increases the segmentation accuracy. We also see that a small number of iterations in the iterative scheme 
    is sufficient for a practical CV task.}
    \label{fig:segmentation_tv}
\end{figure}
But, as the log-loss \eqref{eq:cross_entropy} can be written as a Bregman distance relative to $h$, our primary assumption $l(x,z) \leq D_{E_\theta}(x,z)$ \eqref{eq:blanket_assumption} is fulfilled and we can consider the Bregman surrogate problem in the dual setting of Eq.~\eqref{eq:sur_dual}:
\begin{equation}
    \min_\theta \sum_{i=1}^N \min_{z_i} W_h(\mathcal{N}(\theta,y_i)-z_i, x_i^*) + W_{TV}(z_i,x_i^*),
\end{equation}
which we can rewrite to
\begin{align}\label{eq:bregman_segmentation}
\begin{split}
        \min_\theta \sum_{i=1}^N \min_{||p_i||\leq 1} &h^*\left(\mathcal{N}(\theta,y_i)-D^T p_i\right) \\
        &-\langle \mathcal{N}(\theta,y_i),x_i^*\rangle + ||Dx_i^*||_1.
\end{split}
\end{align}
We note that this is essentially a cross-entropy loss with an additional additive term $p_i$, that is able to balance out incoherent output of $\mathcal{N}(\theta,y_i)$ that would lead to erroneous segmentations with a higher perimeter. Furthermore, the training process is still convex w.r.t to $\mathcal{N}(\theta,y_i)$, in contrast to unrolling schemes. The iterative model \eqref{eq:iterative_scheme} has a very similar structure, including the gradient of the loss into \eqref{eq:bregman_segmentation}. 


To validate this setup, we choose $\mathcal{N}$ to be given by a simple convolutional linear model. We draw a small subset of the \texttt{cityscapes} dataset and compare the cross entropy model of Eq~\eqref{eq:segmentation_simple} with the total variation bi-level model of Eq. \eqref{eq:bregman_segmentation} and its partial and iterative applications. Figure \ref{fig:segmentation_tv} visualizes the training accuracy over training iterations. We find that the proposed approach is able to improve the segmentation accuracy of the linear model significantly. We refer to the appendix for further details.
%
%
\begin{table}
    \centering
    \begin{tabular}{c|c|c|c|c}
        Model  & PSNR & T & PSNR(Iter.)  & TT \\
        \hline
        Total Variation  &  27.41 & - & - & -\\      
        3  3x3 Filters & 26.66  & 00:34 & 27.66 & 02:21 \\      
        48  7x7 Filters &  27.41 & 02:45 & 28.03 & 03:11\\   
        96 9x9 Filters & 27.46 & 01:43 & 28.03 & 02:22\\
    \end{tabular}
    \caption{Training time (T) in minutes for each surrogate computation and PSNR on the test dataset for various gray-scale filters for the energy model in Eq.~\eqref{eq:analysis_model} with and without the iterative process of Eq~\eqref{eq:iterative_surrogate} and total time (TT) for the iterative process are compared to total variation with optimal regularization parameter. Note that training time varies mostly due to differing iteration counts. The results of the convex model of \cite{chen_insights_2014} are reproduced.}
    \label{tab:analyis}
\end{table}

\subsection{Analysis Operator Models}\label{sec:analysis_operators}
Finally, we illustrate the behaviour of our approach on a practically relevant model, learning a set of optimal convolutional filters for denoising \cite{rudin_nonlinear_1992,chen_insights_2014}. We consider the parametric energy model
\begin{equation}\label{eq:analysis_energy_model}
    x(\theta) = \argmin_x \frac{1}{2}||x-y_i||^2 + ||D(\theta) x||_1,
\end{equation}
with $D(\theta)$ denoting the convolution operator to be learned, which is prototypical for many other image processing tasks. We consider square loss $l(x,y) = \frac{1}{2}||x-y||^2$ as a higher loss function and apply our approach. A Bregman surrogate for this model has the form
\begin{equation}\label{eq:analysis_model}
    \min_\theta \sum_{i=1}^N \min_{||p_i||\leq 1} ||D(\theta) x^*_i||_1 + \frac{1}{2}||D^T(\theta) p_i - y_i||^2.
\end{equation}
Model \eqref{eq:analysis_energy_model} was previously considered in \cite{chen_insights_2014,chen_learning_2012}, where it was solved via implicit differentiation. We repeat the setup of \cite{chen_insights_2014} and train a denoising model on the \texttt{BSDS} dataset \cite{martin_database_2001}. Refer to the appendix for the experimental setup and optimization strategy.

Table \ref{tab:analyis} shows both PSNR values achieved when training $D(\theta)$ as convolutional filters as well as training time. In comparison to \cite{chen_insights_2014}, we find strikingly, that we can train a convex model with similar performance to the convex model in \cite{chen_insights_2014}, while being an order of magnitude faster than the original approach. Furthermore in \cite{chen_insights_2014}, the necessary training time jumps from 24 hours for 48 7x7 filters to 20 days for 96 9x9 filters - in our experiment the training time is almost unaffected by the number of parameters, and in this example actually smaller as the larger model converges faster. Also this analysis validates that the iterative process is crucial to reaching competitive PSNR values.

\section{Conclusions}
We investigated approximate training strategies for data-driven energy minimization methods by introducing \emph{parametric majorizers}. We systematically studied such strategies in the framework of convex analysis, and proposed the Bregman distance induced by the lower level energy as well as over-approximations thereof as suitable majorizers. We discussed an iterative scheme that shows promise for applications in computer vision, particularly due to its scalability as shown by its application to image denoising. 

\appendix
\section{Appendix}
\subsection{Convex Analysis in Section 3}
\subsubsection{Details for Derivation of Eqs. (11), (12)}
Eq. (11) above describes the application of Bregman duality:
\begin{equation}
D^0_{E_\theta}(x_i^*,x_i(\theta)) = D_{E_\theta^*}^{x_i^*}(0, q_i) \quad q_i \in \partial E(x_i^*,y_i,\theta) \tag{11},
\end{equation}
which is a common application of the following identity \cite{burger_bregman_2016,benning_modern_2018-1}:
\begin{lemma}[Bregman Identity]\label{app:lemma}
	Consider a convex lsc. function $E:\R^n \to \R$ with a subgradient $p \in \partial E(y)$. Then, the following identity holds:
	\begin{equation*}
	D_E^p(x,y) = D_{E^*}^x(p,q), \quad q \in \partial E(x) 
	\end{equation*}
\end{lemma}
\begin{proof}
	This property follows from equality (Fenchel's identity) in the Fenchel-Young inequality $E(x) + E^*(p) = \langle p,x \rangle \iff p \in \partial E(x)$. To see this we write
	\begin{equation*}
	D^p_E(x,y) = E(x) - \langle p,x\rangle - E(y)  + \langle p,y\rangle 
	\end{equation*}
	and apply Fenchel's identity for $p,y$ to find 
	\begin{equation*}
	D^p_E(x,y) = E(x) - \langle p,x\rangle + E^*(p)
	\end{equation*}
	We then introduce any $q \in \partial E(x)$ by writing $\langle p,x\rangle = \langle p-q+q,x\rangle$ and apply Fenchel's identity again:
	\begin{equation*}
	D^p_E(x,y) = E^*(p) - E^*(q) - \langle x, p -q \rangle = D_{E^*}^x(p,q)
	\end{equation*}
\end{proof}
The step from Eq. (11) to Eq.(12) is simply the first step of this derivation:
\begin{align}
&D_{E_\theta}(x_i^*,x_i(\theta)) &&= E(x_i^*,y_i,\theta) - \langle 0,x_i^*\rangle + E^*(0,y_i,\theta) \notag \\
=& D_{E_\theta^*}^{x_i^*}(0, q_i)&&= E(x_i^*,y_i,\theta) + E^*(0,y_i,\theta) \tag{12}
\end{align}
as $p_i = 0$ is a subgradient of $E$ at $x_i(\theta)$ and $q_i$ at $x_i^*$.

\subsubsection{Details for Derivation of Eq. (14) to (15)}
A crucial subtlety of Lemma \ref{app:lemma} is that this identity holds for any $q \in \partial E(x)$ and the choice of subgradients is irrelevant, the Bregman distance is equal for all choices. This motivates the introduction of the $W$-function $W_E(p,x) = E^*(p) + E(x) - \langle p,x \rangle$. This function is convex in either $p$ or $x$ and always non-negative. It can be understood as measuring the deviation of $p$ from subgradients of $x$ as a direct implementation of the Fenchel-Young inequality. As such it is 0 exactly if $p \in \partial E(x)$. Previous usage of this function can be found for example in \cite{butnariu_proximal-projection_2008,reich_existence_2011}. For Legendre functions \cite{bauschke_legendre_1997}, i.e. functions where both $E$ and $E^*$ are (essentially) smooth, the connection to Bregman distances is immediate:
\begin{equation*}
W_E(p,x) = D_E^p(x, \nabla E^*(p)),
\end{equation*}
for non-smooth functions this is also a part of the proof of Lemma \ref{app:lemma}, replacing $\nabla E^*(p)$ by $y \in \partial E^*(p)$. As such, we can write Eq. (12) as
\begin{equation}
D_{E^*}^{x_i^*}(0,q_i) = W_{E_\theta}(0, x_i^*) \tag{12}.
\end{equation}
The introduction of this function then allows us to show that
\begin{equation}
W_{E}(0, x_i^*) = \min_z W_{E_1,\theta}(-z,x_i^*) + W_{E_2, \theta}(z,x_i^*) \tag{15}
\end{equation}
under the assumption in Eq.(13), that $E$ can be written as $E_1 + E_2$, with both functions convex. We recognize this as the clear extension of the infimal convolution property $E^*(0) = \min_z E_1^*(-z) + E_2^*(z)$ (which itself can be understood as Fenchel's duality theorem applied to $E_1$, $E_2$) to these functions, in the smooth setting this could be written via
\begin{align*}
D_{E^*}^{x_i^*}(0,\nabla E(x_i^*)) = \min_z ~ ~ & D_{E_1^*}(-z,\nabla E_1(x_i^*)) \\
+& D_{E_2^*}(z,\nabla E_2^*(x_i^*)).
\end{align*}
We arrive at Eq. (15) from Eq. (14) by rewriting $E$ in Eq.(14):
\begin{align}
\begin{split}
&\min_z E_1(x_i^*,y_i,\theta) + E_2(x_i^*,y_i,\theta) \\
&+ E_1^*(-z,y_i,\theta) + E_2^*(z,y_i,\theta)    
\end{split} \tag{14}\\
\begin{split}
=  &\min_z  E_1(x_i^*,y_i,\theta) + E_2(x_i^*,y_i,\theta) + \langle z, x_i^* \rangle \\
&+ E_1^*(-z,y_i,\theta) + E_2^*(z,y_i,\theta) - \langle z, x_i^* \rangle
\end{split}\notag \\
= &\min_z W_{E_1,\theta}(-z,x_i^*) + W_{E_2,\theta}(z,x_i^*) \tag{15}.
\end{align}

\subsubsection{Proof of Proposition 2}
\begin{figure*}
	\centering
	\includegraphics[width=0.45\textwidth]{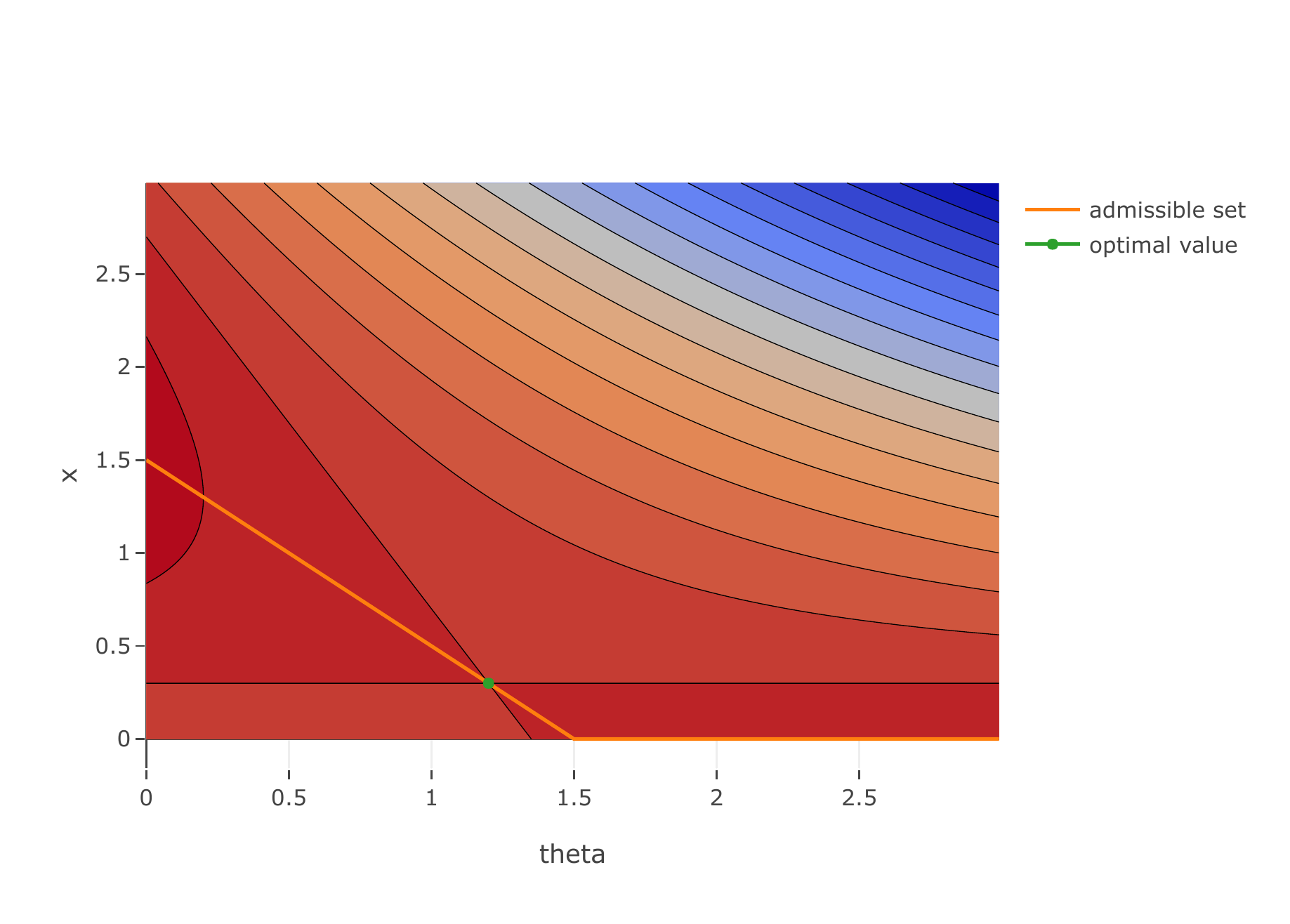}
	\includegraphics[width=0.45\textwidth]{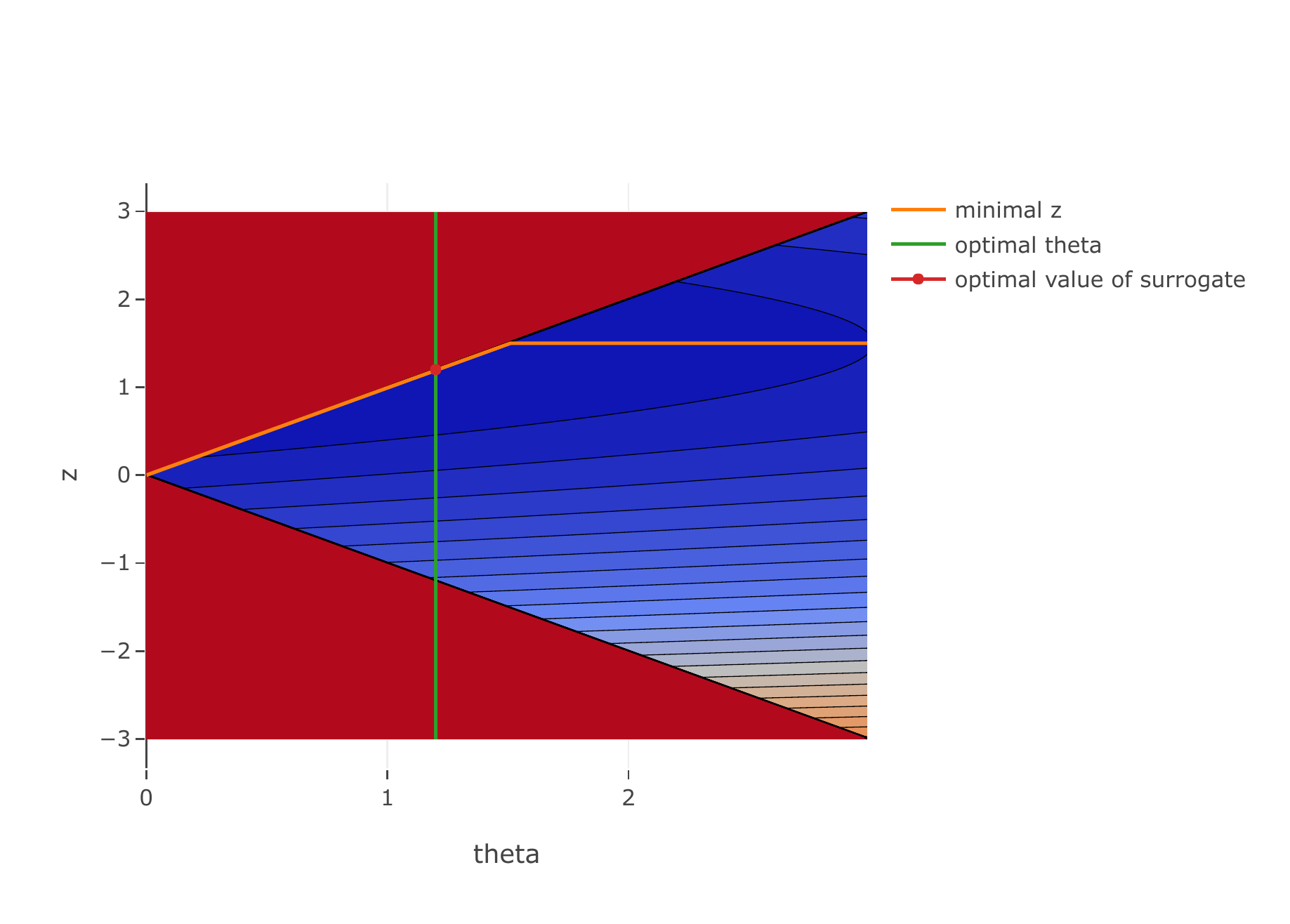}
	\caption{Visualization of the Bregman surrogate problem in primal formulation (left) and dual formulation (right). The problem in visualized over all $(x,\theta)$, respectively $(z,\theta)$. The admissible $x(\theta)$ are marked in orange in the left contour plot and the optimal $z(\theta)$ one the right. The optimal value in $\theta$ is marked in green in both plots.}
	\label{fig:app1}
\end{figure*}
\setcounter{proposition}{1}
\begin{proposition}[Ordering of parametric majorizers]
	Assuming the condition $l(x,z) \leq D_{E_\theta}(x,z)$ from Eq.~(8), we find that the presented parametric majorizers can be ordered in the following way:
	\begin{align*}
	l(x_i^*,x(\theta)) &\leq D_{E_\theta}^0 \left( x_i^*,x_i(\theta) \right) = D_{E^*_\theta}^{x_i^*}(0, q_i) \\
	&\leq \min_{z \in \partial E_2(x_i^*)} W_{E_1}(-z,x_i^*) \\
	& \leq \frac{1}{m(\theta,y)} ||q_i||^2 \st q_i \in \partial E(x_i^*,y,\theta).
	\end{align*}
	The Bregman surrogate (10) majorizes the original loss function and is in turn majorized by the partial surrogate (16) which is majorized by the gradient penalty (17) under the assumption of $m(\theta,y)$ - strong convexity of $E_1$.
\end{proposition}
\begin{proof} The first inequality follows directly by the assumption $l(x,z) \leq D_{E_\theta}(x,z)$. The second inequality is the application of Bregman Duality discussed in Lemma \ref{app:lemma}. From Eq.(15) we now see that $D_{E_\theta}^{x_i^*}(0,q_i)$, $q_i \in \partial E(x_i^*,y_i,\theta)$ can be written as a minimum over $z$. Clearly choosing a non-optimal $z$ yields an upper bound to this minimal value. Without loss of generality, we choose $z \in \partial E_2(x_i^*)$ so that $W_{E_2,\theta}(z,x_i^*)$ is equal to zero. 
	
	Now we assume that $E$ is $m(\theta,y)$-strongly convex. We subsume this strong convexity term in $E_1$ again without loss of generality so that $E_1$ is strongly convex. By convex duality \cite{bauschke_convex_2011}, this implies that $E_1^*$ is $m(\theta,y)$ strongly smooth, i.e. $D^x_{E_1^*}(p,q) \leq \frac{1}{2 m(\theta,y)}||p-q||^2$. Following Eq.(12), we write
	\begin{alignat*}{2}
	&W_{E_1^*}(-z,x_i^*) = D_{E_1^*}^{x_i^*}(-z,r) \quad &z \in \partial E_2(x_i^*,y_i,\theta),\\ 
	&  &r \in \partial E_1(x_i^*,y_i,\theta) \\
	&\leq \frac{1}{2m(\theta,y)}||-z - r||^2  &\\
	&= \frac{1}{2m(\theta,y)}||q_i||^2 \quad &q_i \in \partial E(x_i^*,y_i,\theta),
	\end{alignat*}
	under mild assumptions on the additivity of subgradients of $E_1$ and $E_2$.
\end{proof}

\subsubsection{Derivation of the surrogate functions for the example in subsection 3.3}
Section 3.3 discusses the non-smooth bi-level problem given in Eqs. (18) and (19):
\begin{align}
& \min_{\theta \in \R} \frac{1}{2}|x^*-x(\theta)|^2, \tag{18}  \\
\textnormal{subject to }\qquad & x(\theta) = \argmin_x \frac{1}{2}|x-y|^2 + \theta |x|. \tag{19}
\end{align}
for both $x^*,y \in \R$. In this setting, the 'primal' formulation of the Bregman surrogate is given by
\begin{equation}
\min_\theta \max_x \frac{1}{2}|x^*-y|^2 - \frac{1}{2}|x-y|^2 + \theta\left(|x^*|-|x| \right) \tag{10 ex.}
\end{equation}
whereas the 'dual' formulation is given by 
\begin{equation}
\min_\theta \min_{|z| \leq \theta} \frac{1}{2}|x^*-y|^2 + \theta|x^*| + \frac{1}{2}|z-y|^2 \tag{12 ex.}.
\end{equation}
Note that this problem is convex in $z,\theta$ as the epigraph constraint $|z|\leq \theta$ is convex.
Both (equivalent!) variants are visualized in Figure \ref{fig:app1}. We see that the saddle-point of the primal formulation and the minimizer of the dual formulation correctly coincide with the optimal $\theta$.

Moving forward, we set $E_1(x,y) = \frac{1}{2}|x-y|^2$ and $E_2(x,\theta) = \theta|x|$ to compute the two partial surrogates. Firstly $W_{E_1,\theta}(-z,x^*), z \in \partial E_2(x^*)$ leads to
\begin{equation}
\min_\theta \frac{1}{2}|x^* -y + q|^2, \quad q \in \partial |x^*|, \tag{16 ex.1}
\end{equation}
where we take $q = \operatorname{sign}(x^*)$ as $x^* \neq 0$ in our example. As $E_1$ is a quadratic function, this is also equivalent to the gradient penalty in Eq. (17).
The second partial surrogate, $W_{E_2,\theta}(z,x^*), z \in \partial E_1(x^*)$ can be written as
\begin{align}
&\min_\theta \theta |x^*| + I_{|\cdot|\leq \theta}(x^*-y) - \langle x^*,x^* -y \rangle \tag{16 ex.2} \\
=&\min_{|x^*-y|\leq \theta} \theta |x^*| + C \notag.
\end{align}
Figure \ref{fig:app1} here and Figure 1 in the main paper both arise from the data point $x^* = 0.3, y = 1.5$.

To give some more details on the fact that the Bregman surrogate is exactly identical with the original loss function in the vicinity of the optimal value, note that this is caused by the special structure of the Bregman distance of the absolute value, $D_{|\cdot|}(x,y)$ as $D_{E_\theta}(x,y)$ decomposes into $\frac{1}{2}|x-y|^2 + \theta D_{|\cdot|}(x,y)$. This function is equal to the higher-level loss function as soon as the signs of $x^*$ and $x(\theta)$ coincide and as such the majorizer is exact, even if it is much easier to compute.

\subsubsection{Proof of Proposition 4}
Section 3.4 describes an iterative procedure for repeated application of the majorization strategies discussed in section 3.2. This scheme was based on the result of Proposition 3:
\begin{equation}
l(x,y) \leq l(x,z) + \langle \nabla_z l(x,z), y- z \rangle + D_E(z,y) \tag{20},
\end{equation}
inserting $x=x_i^*, y=x_i(\theta), z = x_i(\theta^k)$ leads to
\begin{align}
\begin{split}
l(x_i^*,x_i(\theta)) \leq l(x_i^*,x_i(\theta^k)) +D_{E_\theta}(x_i(\theta^k),x_i(\theta))\\
+ \langle \nabla l(x_i^*,x_i(\theta^k)), x_i(\theta)- x_i(\theta^k) \rangle.
\end{split}\tag{20b}
\end{align}

\setcounter{proposition}{3}
Eq.(20), respectively (20b), lead to a monotone descent of the higher-level loss, as shown in Proposition 4:
\begin{proposition}[Descent Lemma]
	The iterative procedure given by
	\begin{align*}
	\theta^{k+1} = \argmin_\theta &\sum_{i=1}^N l(x_i^*,x_i(\theta^k)) \\
	&+ \langle \nabla l(x_i^*,x_i(\theta^k)), x_i(\theta) - x_i(\theta^k) \rangle \\
	&+ D^0_{E_\theta}(x_i(\theta^k),x_i(\theta))
	\end{align*}
	is guaranteed to be stable, i.e. not to increase the bi-level loss:
	\begin{equation}
	\sum_{i=1}^N l \left( x_i^*,x_i(\theta^{k+1}) \right) \leq \sum_{i=1}^N l \left( x_i^*,x_i(\theta^{k}) \right) \tag{23}
	\end{equation}
\end{proposition}

\begin{proof}[Proof of Proposition 4]
	$\theta^{k+1}$ is a minimizer of the iterative scheme. Therefore, evaluating the iteration at $\theta^{k+1}$ leads to a lower value than evaluating at $\theta^k$:
	\begin{align*}
	\sum_{i=1}^N &l(x_i^*,x_i(\theta^k)) 
	+ \langle \nabla l(x_i^*,x_i(\theta^k)), x_i(\theta^{k+1}) - x_i(\theta^k) \rangle \\
	+ &D^0_{E_{\theta^{k+1}}}(x_i(\theta^k),x_i(\theta^{k+1})) \\
	\leq     \sum_{i=1}^N &l(x_i^*,x_i(\theta^k)) 
	+ \langle \nabla l(x_i^*,x_i(\theta^k)), x_i(\theta^k) - x_i(\theta^k) \rangle \\
	+ &D^0_{E_{\theta^k}}(x_i(\theta^k),x_i(\theta^k)) \\
	= \sum_{i=1}^N &l(x_i^*,x_i(\theta^k)) 
	\end{align*}
	Now the left-hand-side is also equivalent to Eq. (20b) evaluated at $\theta^{k+1}$. Applying the inequality in (20b) for all $i=1,\dots,N$ we find
	\begin{align*}
	\sum_{i=1}^N l(x_i^*,x_i(\theta^{k+1})) \leq \sum_{i=1}^N l(x_i^*,x_i(\theta^k)). 
	\end{align*}
\end{proof}

\begin{remark}
	The iterative scheme given in Eq.(22), i.e.
	\begin{align}
	\begin{split}
	\theta^{k+1} = \argmin_\theta \sum_{i=1}^N& E^* \left( \nabla l(x_i^*,x_i(\theta^k)),y_i,\theta\right) \\
	+& E \left(x(\theta^k),y_i,\theta \right).       
	\end{split}\tag{22}
	\end{align}
	is an over-approximation of the iterative scheme discussed in Proposition 4. As such we expect the results of Proposition 4 to hold only approximately as stated in the main paper. 
\end{remark}
\begin{figure*}
	\centering
	\begin{tabular}{cccc}
		\includegraphics[trim={0.5cm 0.1cm 0.2cm 0.85cm},clip,height = 0.17\textwidth]{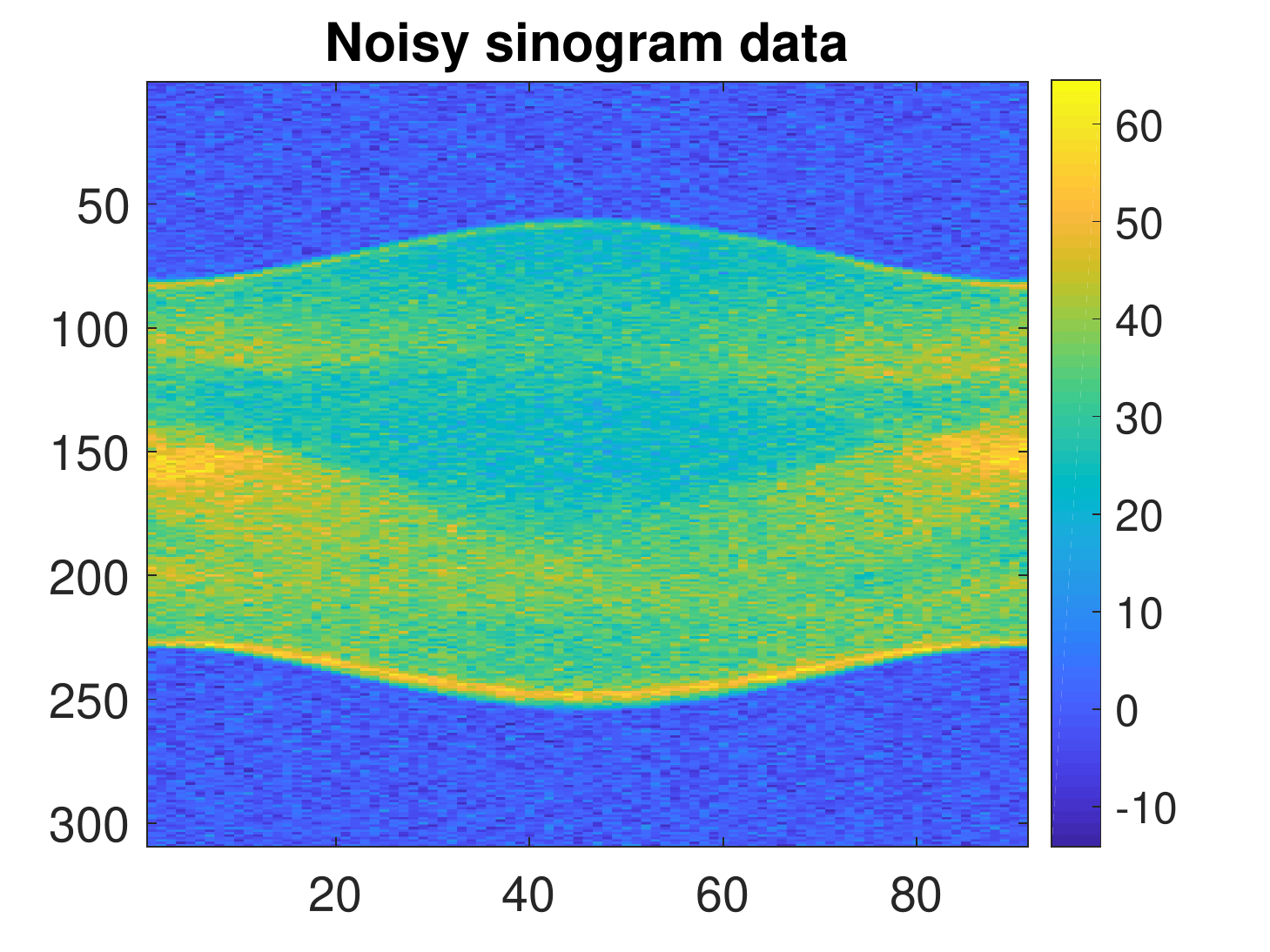} &
		\includegraphics[trim={0.5cm 0.1cm 0.2cm 0.8cm},clip,height = 0.17\textwidth]{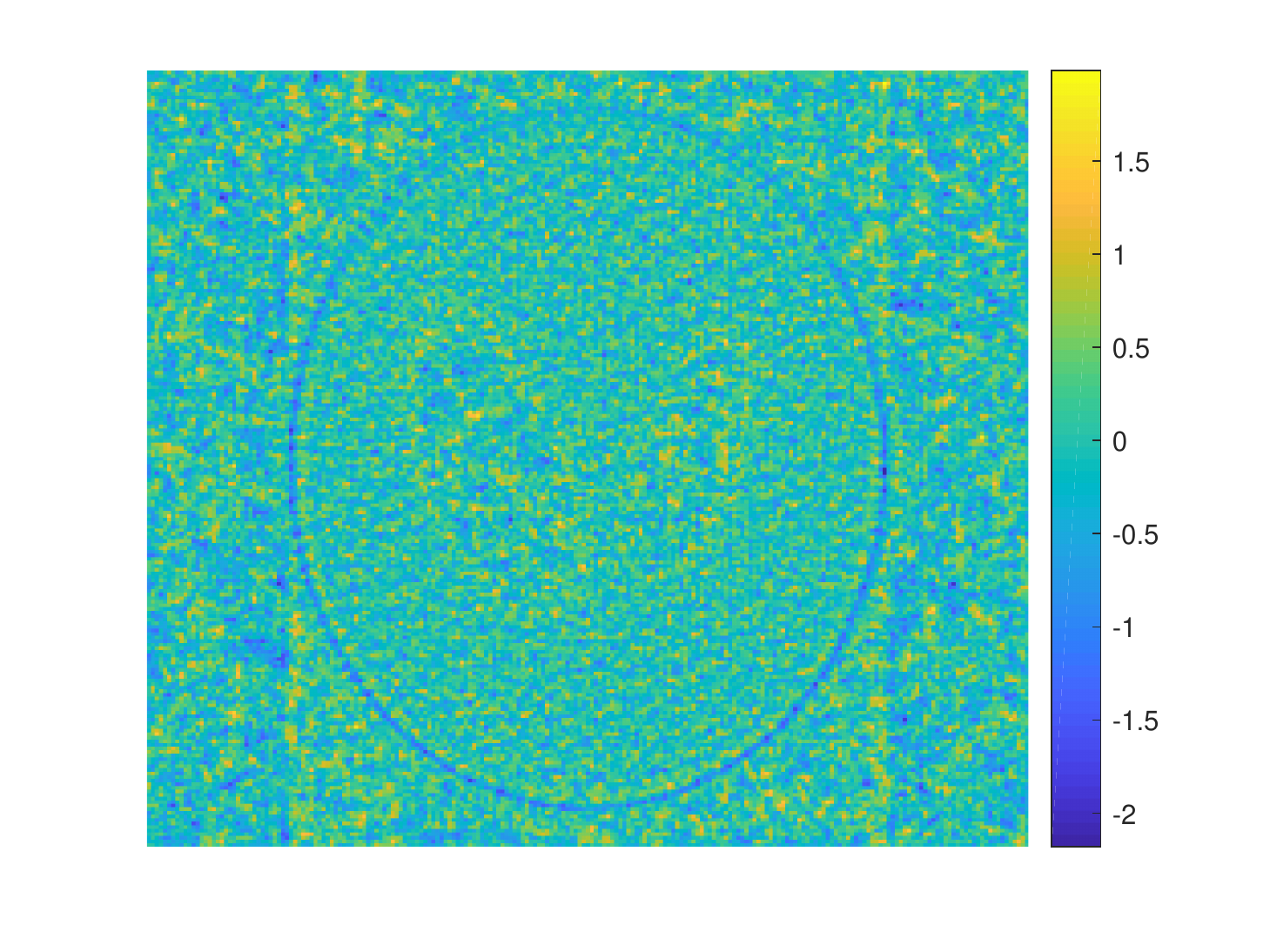} &
		\includegraphics[height = 0.17\textwidth]{zoomedOriginal.png} &
		\includegraphics[height = 0.17\textwidth]{zoomedxLearned.png}\\
		\small Noisy sinogram $y$ & \small correction term $ \mathcal{N}(\theta,y_i)$ &\small  Huber-TV, PSNR 23.9 &\small  Learned correction, PSNR 25.8\\
	\end{tabular}
	\caption{Illustrate our results for learning a linear correction term for a Huber-reguralized CT reconstruction problem. In reference to Figure 2 in the main paper we also visualize input data and the learned linear correction map. The predicted linear correction term can be visualized and inspected, and its influence can easily be quantified or explicitly scaled via a parameter.}
	\label{fig:ctExample2}
\end{figure*}    
\subsection{Experimental Setup}
This section will add additional details to the experiments presented in the paper\footnote{Refer also to the implementations hosted at \url{https://github.com/JonasGeiping/ParametricMajorization}}. 

\subsubsection{CT - Additional Details}
The implementation of the CT example in section 4.1 is straightforward. We generate pairs $(y_i^*,x_i^*)$ of noisy sinograms and ground truth images and optimize 
\begin{align*}\label{eq:ctPrimalSurrogate}
\min_{\theta \in \R^p} \sum_{i=1}^n \|A^*Ax_i^* - A^*y_i + \beta \nabla R(x_i^*) + \mathcal{N}(\theta, y_i)\|_2^2.
\end{align*}
We test our model on
the widely-used Shepp-Logan phantom, comparing the learned model with a pure Huber-TV solution, for which we found the optimal parameter $\beta$ by grid search. This setup was implemented in Matlab. To visualize the linear correction term, we repeat an extended version of Figure 2 in Figure \ref{fig:ctExample2}.

\subsubsection{Segmentation - Additional Details}
The segmentation experiment shown in Figure 3 of the main paper shows the results of training the variational model in Eq.(25), which corresponds to an augmented cross-entropy term, as discussed in section 4.2. 

The partial surrogate implemented in Figure 3 is a direct application of Eq.(16) to the segmentation setting, giving
\begin{equation*}
\min_\theta \sum_{i=1}^N \min_{p_i \in \partial ||Dx_i^*||} D_h \left( x_i^*,\nabla h^* \left( \mathcal{N}(\theta,y_i)-D^T p_i\right) \right),
\end{equation*}
where the computation of the auxiliary variable $p_i$ is simplified. Note further that the gradient penalty cannot be applied in this setting, as the segmentation energy $E$ is not strongly convex. Similarly, the iterative approach can be computed to be
\begin{align*}
\begin{split}
\min_\theta \sum_{i=1}^N \min_{||p_i||\leq 1} &h^*\left(\frac{x_i^*}{x_i(\theta^k)} + \mathcal{N}(\theta,y_i) -D^T p_i\right) \\
&-\left\langle \mathcal{N}(\theta,y_i),x_i(\theta^k)\right\rangle 
\end{split}
\end{align*}
which is still convex in $\mathcal{N}(\theta,y)$, but the input arguments now take previous solutions into account. 

To emphasize the convexity of the setup, we choose $\mathcal{N}(\theta,y_i)$ as a linear convolutional network of $3x3x3$ filters for each target class. We accordingly optimize the resulting convex minimization problems  by an optimal convex optimization method, namely FISTA \cite{beck_fast_2009}. To solve the inference problem in Eq. (25) we apply usual strategies and optimize via a primal-dual algorithm \cite{chambolle_first-order_2011} - to increase the speed we adapt a recent variant \cite{chambolle_ergodic_2016} and consider the Bregman-Proximal operator in the primal sub-problem for which we use the entropy function $h$ described in the paper, paralleling \cite{beck_mirror_2003,ochs_techniques_2016}.

We draw four images and their corresponding segmentations from the \texttt{cityscapes} data set \cite{cordts_cityscapes_2016} and implement the proposed procedures in PyTorch \cite{paszke_automatic_2017}. For Figure 3 we drew the first four images, which we resized to 128x256 pixels. To visualize the improvement over the iterations, we initialize the subsequent iterations of the iterative scheme again with the initial value of $\theta$, so that the training accuracy curves in Figure 3 are comparable. This is of course not strictly necessary and $\theta$ could be initialized with the current estimate in every iteration. We also point out that we visualize the actual training accuracy in Figure 3, meaning the percentage of successfully segmented pixels after \textit{hard argmax} of the results of the algorithms.
\subsubsection{Analysis Operators - Additional Details}
For this experiment we considered the task of learning an 'analysis operator' $D(\theta)$, i.e. a set of convolutional filters $\theta^k$ so that $D(\theta) = \sum_{k=1}^K \theta_k * x $ for a set of $K$ filters. Due to anisotropy, we can write the resulting minimization problem as
\begin{equation*}
x(\theta) = \argmin_x \frac{1}{2}||x-y||^2 + \sum_{k=1}^K ||\theta_k * x||_1.
\end{equation*}
We repeat the experimental setup of \cite{chen_insights_2014} and train this model on image pairs $x^*,y$ of noise-free and noisy image patches, to learn filters that result in a convex denoising model \cite{chen_bi-level_2014-1,chen_insights_2014}. To do so we draw a batch of 200 $64x64$ image patches from the training set of the Berkeley Segmentation data set \cite{martin_database_2001}, convert the images to gray-scale  and add Gaussian noise. To compare with \cite{chen_insights_2014} and \cite{zhang_beyond_2017} we do not clip the noisy images and use Matlab's \texttt{rgb2gray} routine to generate this data. Further, as in \cite{chen_insights_2014}, we do not optimize directly for the convolutional filters, but instead decompose each filter into a DCT-II basis, where we learn the weight of each basis function, excluding the constant basis function \cite{huang_statistics_1999}. Before training we initialize these weights by orthogonal initialization \cite{saxe_exact_2013} with a factor of $0.01$, respectively $0.001$ for the larger 9x9 filters. 

To solve the training problem we minimize Eq. (33) in the paper jointly in $\theta, \lbrace p_i\rbrace_{i=1}^N$. We do this efficiently by taking steps toward the optimal weights with the 'Adam' optimization procedure \cite{kingma_adam:_2015} with a step size $\tau = 0.1$ (although gradient descent with momentum or FISTA \cite{beck_fast_2009} are also valid options). We use a standard accelerated primal-dual algorithm \cite{chambolle_first-order_2011} to solve the convex inference problem.
For the iterative procedure we repeat this process, computing $x(\theta^k)$ after every minimization of Eq.(33), inserting it as a factor into $E^*$ and repeating the optimization. If the iterative procedure increases the loss value, we reduce the step size $\tau$ of the majorizing problem and repeat the step. If reducing the step size does not successfully improve the result for several iterations, we terminate the algorithm.

We implement this setup in PyTorch \cite{paszke_automatic_2017} and refer to our reference implementation for further details.

For total variation denoising, which corresponds to choosing $D(\theta)$ as the gradient operator with appropriate scaling, $\alpha\nabla$, we use grid search to find the optimal scaling parameter $\alpha$.

We report execution times for a single minimization of Eq.(33) for different filter sizes in Table 1 in the paper as well as total time for an iterative procedure. These timings are reported for a single \textit{GeForce RTX 2080Ti} graphics card.

{\small
\bibliographystyle{ieee_fullname}
\bibliography{zotero_library}
}

\end{document}